\documentclass[11pt]{amsart}
\usepackage{amsmath,amsfonts,amscd,amssymb,amsthm,epsfig,euscript}
\usepackage{mathrsfs}
\usepackage{amsxtra}
\usepackage[all]{xy}
\usepackage{verbatim,epsf}
\newtheorem{theorem}{Theorem}
\newtheorem{proposition}{Proposition}
\newtheorem{lemma}{Lemma}
\newtheorem{corollary}{Corollary}
\theoremstyle{definition}
\newtheorem{definition}{Definition}

\theoremstyle{remark}
\newtheorem{remark}{Remark}
\numberwithin{equation}{section}
\newcommand{\field}[1]{\ensuremath{\mathbb{#1}}}
\newcommand{\CC}{\field{C}}

\newcommand{\HH}{\field{H}}
\newcommand{\PP}{\field{P}}
\newcommand{\RR}{\field{R}}

\newcommand{\fd}{\mathfrak{d}}
\newcommand{\fe}{\mathfrak{e}}
\newcommand{\fg}{\mathfrak{g}}
\newcommand{\fh}{\mathfrak{h}}
\newcommand{\fk}{\mathfrak{k}}

\newcommand{\fm}{\mathfrak{m}}
\newcommand{\fo}{\mathfrak{o}}

\newcommand{\bv}{\mathbf{v}}

\newcommand{\bfe}{\mathbf{e}}
\newcommand{\bfl}{\mathbf{l}}
\newcommand{\bfI}{\mathbf{I}}
\newcommand{\curly}[1]{\mathscr{#1}}

\newcommand{\cC}{\curly{C}}

\newcommand{\cL}{\curly{L}}

\newcommand{\cR}{\curly{R}}

\newcommand{\cU}{\curly{U}}

\newcommand{\SO}{\mathrm{{SO}}}

\newcommand{\OO}{\mathcal{O}}

\DeclareMathOperator{\tr}{tr}

 \DeclareMathOperator{\ad}{ad}

\usepackage{hyperref}
\hypersetup{colorlinks,citecolor=blue,plainpages=false,hypertexnames=false}
\begin{document}
\title[Deformations with angular momentum]{Linear phase space deformations with angular momentum symmetry}

\author[Meneses]{Claudio Meneses}
\address{Mathematisches Seminar, Christian-Albrechts Universit\"at zu Kiel, Ludewig-\indent Meyn-Str. 4, 
24118 Kiel, Germany}
\email{meneses@math.uni-kiel.de} 


\maketitle
\begin{abstract}
Motivated by the work of Leznov--Mostovoy \cite{LM-03}, we classify the linear deformations of standard $2n$-dimensional phase space 
that preserve the obvious symplectic $\mathfrak{o}(n)$-symmetry. As a consequence, we describe standard phase space, as well as $T^{*}S^{n}$ and $T^{*}\mathbb{H}^{n}$ with their standard symplectic forms, as degenerations of a 3-dimensional family of coadjoint orbits, which in a generic regime are identified with the Grassmannian of oriented 2-planes in $\mathbb{R}^{n+2}$.\\

\end{abstract}

\section{Introduction and statement of the problem}

The notions of momentum maps and symplectic reduction provide a very convenient formulation of integrability in classical mechanics \cite{GS84, GS06}. A standard example is the integrability of central potential Hamiltonians such as the $n$-dimensional Kepler and harmonic oscillator problems, which can be understood in terms of dynamical (i.e. symplectic) symmetries for the groups $\mathrm{SO}(n+1)$ and $\mathrm{SU}(n)$ \cite{Len24,Fra65}. These principles can be exploited in more general contexts, such as the analogous integrability of the Kepler and harmonic oscillator problems on the round $n$-sphere studied by Higgs \cite{Higgs78}. The standard formulation of these problems can then be recovered from a limiting process, if we interpret the sectional curvature of the $n$-sphere as a deformation parameter yielding a commutative linear deformation of the Poisson structure in standard phase space.

We will address a natural generalization of the previous idea. A precise formulation relies on three facts (two of which are proved in appendix \ref{appendix}):
\begin{itemize}
\item[(i)] $\left(\mathbb{R}^{n}\oplus\mathbb{R}^{n},\omega = \sum_{i = 1}^{n} dx_{i}\wedge dp_{i}\right)$ is symplectomorphic to a connected component $\OO^{+}_{2n}$ of a coadjoint orbit $\OO_{2n}$ of $\mathrm{G}_{n} = \mathrm{O}(n)\ltimes \mathrm{H}_{n}$, where $\mathrm{H}_{n}$ is the $2n+1$-dimensional Heisenberg group. The action of $\mathrm{O}(n)$ on $\mathrm{H}_{n}$ is the standard one on its $\RR^{n}\oplus\RR^{n}$ subgroup and trivial on the central extension element.
\item[(ii)] $\fg_{n} = \fo(n)\ltimes\fh_{n}$ is the orthogonal Lie algebra associated to a quadratic form $Q_{0}$ on $\RR^{n+2}$ of isotropy index 2 and signature $(n,0)$, and $\OO_{2n}$ is identified with a Zariski open set in $\widetilde{\mathrm{Gr}}_{2}\left(\RR^{n+2}\right)$, the Grassmannian of oriented 2-planes in $\RR^{n+2}$.
\item[(iii)] $(T^{*}S^{n},\omega)$ and $(T^{*}\mathbb{H}^{n},\omega)$ with their standard symplectic forms are symplectomorphic to coadjoint orbits (a connected component in the latter case) for \emph{deformations} of $\fg_{n}$ that are respectively isomorphic to $\mathfrak{e}(n + 1) = \fo(n+1)\ltimes \RR^{n+1}$ and  $\mathfrak{e}(n, 1) = \fo(n,1)\ltimes \RR^{n+1}$. The deformation parameter is interpreted as the sectional curvature of an underlying configuration space (remark \ref{rem:cotangent}). 
\end{itemize}

The problem that we pose is the following: \emph{to classify all deformations of the Lie algebra $\fg_{n}$ and the subsequent coadjoint orbits that specialize to the previous examples}.\footnote{The study of deformations and contractions of Lie algebras in physics originates in \cite{IW53}. The reader can find a leisurely exposition of such ideas in \cite{GS84,GS06}.}

At a technical level, the problem is equivalent to the understanding of the Lie algebra cohomology $H^{2}(\fg_{n},\fg_{n})$. In virtue of (ii), the problem is related to the study of equivalence classes of deformations of a degenerate quadratic form in $\RR^{n+2}$ of isotropy index 2 and signature $(n,0)$. Our treatment emphasizes the physical and geometric features of the problem while solving it, and is motivated by the work of Leznov-Mostovoy \cite{LM-03}. They studied the Kepler problem on a 3-dimensional family of deformations of standard phase space in the special case $n = 3$, interpreted as commutative deformations of the standard Poisson structure on $C^{\infty}\left(\RR^{3}\oplus\RR^{3}\right)$. The main result of this article is a proof that the generators of the Leznov-Mostovoy deformations determine three cocycles spanning $H^{2}(\fg_{n},\fg_{n})$ for arbitrary $n \geq 3$.

\begin{theorem}\label{def-Heisenberg}
Let $n \geq 3$. 
The space of infinitesimal deformations of $\fg_{n}^{\CC}$ is three-dimensional. Every infinitesimal deformation is integrable, and the generic deformation is isomorphic to $\fo(n+2,\CC)$.
\end{theorem}

However, the induced linear deformations of $\fg_{n}^{\CC}$ are not independent (corollary \ref{cor:effective-parameters}), marking a subtle difference between the infinitesimal and global pictures. The effective family of linear deformations of $\fg_{n}^{\CC}$ turns out to be at most two-dimensional. 

The phase space deformations that we will study are natural generalizations of the phase space of a manifold of constant sectional curvature, and contain the latter as particular cases (remark \ref{rem:cotangent}). Such deformations can be understood geometrically in terms of the Grassmannian of oriented planes in $\RR^{n+2}$, relative to a family of deformations of  a degenerate quadratic form of signature $(n,0)$ and isotropy index 2. In particular, there is a generic regime of deformation parameters for which the induced symplectic manifold is \emph{compact}. This leads to the possibility of extending the study of classical and quantum integrable systems on spaces of constant curvature to the most general phase space deformations that preserve a notion of angular momentum symmetry, by means of the study of the geometry of suitable momentum maps.  We plan to address such a problem in the future.

The work is organized as follows. Section \ref{sec:deformations} is dedicated to presenting a proof of theorem \ref{def-Heisenberg} following an application of the Hochschild-Serre spectral sequence. An argument for an alternative proof in terms of the geometry of quadratic forms is given in remark \ref{rem:def-quadratic form}. The rest of the article is a series of applications of theorem \ref{def-Heisenberg}. Section \ref{sec:coadjoint} describes the general family $\OO_{2n}(\pmb{\varepsilon})$ of coadjoint orbits, induced by deformations of $\fg_{n}$, that correspond to the deformations of standard phase space (corollary \ref{cor:deformations}). Section \ref{sec:free motion} describes some geometric structures in $\OO_{2n}(\pmb{\varepsilon})$ corresponding to the induced deformations of the position and momentum polarizations in phase space, and the Euclidean group momentum map corresponding to the free-motion Hamiltonian.

\section{Deformations of the Lie algebra $\mathfrak{o}(n)\ltimes \mathfrak{h}_{n}$}\label{sec:deformations}

It will be convenient to work with complex Lie algebras, since the real deformations of $\fg_{n}$ can be regarded as all possible real forms of a complex deformation of its complexification. The infinitesimal deformations of a complex Lie algebra $\fg$ are described by the Lie algebra cohomology space $H^{2}(\fg,\fg)$ with respect to the adjoint representation \cite{CE48,Ger64,LN67,Fia85}. We are interested in the case when $\fg$ is a semidirect product of a \emph{semisimple} Lie subalgebra $\fk\subset\fg$ and an ideal $\fh\subset\fg$, i.e. $\fg = \fk\ltimes \fh$. Hence $\fg$ is also a module for $\fk$ and $\fh$. Let $E^{p,q}_{2} = H^{p}\left(\fk,H^{q}(\fh,\fg)\right)$. It follows from Whitehead's lemma that  $E^{1,1}_{2} = 0$ and $E^{2,0}_{2} = 0$. The Hochschild-Serre spectral sequence \cite{HS-53,Weibel} collapses from the $E_{2}$-term for $p+q = 2$. 
Hence, restriction induces the isomorphism
\begin{equation}\label{iso-cohomology}
H^{2}(\fg,\fg)\cong E^{0,2}_{2}\cong H^{2}(\fh,\fg)^{\fk}
\end{equation} 
(cf. \cite[theorem 13]{HS-53}). In order to describe the deformations of $\fg_{n}^{\CC}$, we will first consider the case $\fk=\fo(n,\CC)$, $\fh=\CC^{n}$ so $\fg=\fe^{\CC}_{n}=\fo(n,\CC)\ltimes \CC^{n}$, the complexification of the Euclidean Lie algebra in dimension $n \geq 3$. 
In terms of a canonical basis $\{\bfe_{i},\bfl_{jk}\}$,  $\fe^{\CC}_{n}$ is defined by the commutation relations
\[
[\bfe_{i},\bfe_{j}]=0,\quad [\bfl_{ij},\bfe_{k}]=\delta_{ik}\bfe_{j}-\delta_{jk}\bfe_{i},\quad
[\bfl_{ij},\bfl_{kl}]=\delta_{ik}\bfl_{jl}+\delta_{jl}\bfl_{ik}
-\delta_{il}\bfl_{jk}-\delta_{jk}\bfl_{il}.
\]
We will now describe the space $H^{2}\left(\CC^{n},\fe_{n}^{\CC}\right)^{\fo(n,\CC)}$ explicitly. By definition, a 2-cocycle is a linear map $f:\bigwedge^{2}\CC^{n}\to\fe^{\CC}_{n}$
 satisfying 
 \[
 [\bfe_{i},f(\bfe_{j},\bfe_{k})]-[\bfe_{j},f(\bfe_{i},\bfe_{k})]+[\bfe_{k},f(\bfe_{i},\bfe_{j})]=0\quad \forall i,j,k.
 \]
A 2-coboundary is a linear map of the form 
\[
f(\bfe_{i},\bfe_{j})=[l(\bfe_{i}),\bfe_{j}]+[\bfe_{i},l(\bfe_{j})]
\]
for some linear map $l:\CC^{n}\to\fe^{\CC}_{n}$. A 2-cocycle $f$ is called \emph{invariant} if $\forall g\in\fo(n,\CC)$,
 \begin{equation}\label{cocycle-inv}
(g\cdot f)(\cdot,\cdot) = [g,f(\cdot,\cdot)]-f([g,\cdot],\cdot)-f(\cdot,[g,\cdot])=\text{a coboundary}.
 \end{equation}

 \begin{lemma}\label{def-Euclidean}
The space of infinitesimal deformations of the complexification of the Euclidean Lie algebra $\fe^{\CC}_{n}$, $n\geq 3$, is one-dimensional and generated by the invariant 2-cocycle
\[
f(\bfe_{i},\bfe_{j})=\bfl_{ij}.
\]
Every infinitesimal deformation is integrable. Together, they determine a one-dimensional family of Lie algebras $\fe^{\CC}_{n}(\varepsilon)$, where $\fe^{\CC}_{n}(\varepsilon)\cong\fo(n+1,\CC)$ for $\varepsilon\neq 0$ (cf. \cite{GS84}).
\end{lemma}
\begin{proof}
Any 2-coboundary necessarily takes values in $\CC^{n}$. Hence, any two invariant 2-cocycles taking values in $\fo(n,\CC)$ are cohomologous if and only if they are equal. 

That $f$ is an invariant element in $Z^{2}\left(\CC^{n},\fe_{n}^{\CC}\right)$ under the $\fo(n,\CC)$-action is a routine computation. By the previous remark, the cohomology class of $f$ is nontrivial.

It remains to show that up to a constant, this is the only possibility for $f$. First, let us assume that $\forall i,j,$ $f(\bfe_{i},\bfe_{j})\in\fo(n,\CC)$, then the invariance condition \eqref{cocycle-inv} implies that for any $g\in\fo(n,\CC)$ such that $[g,\bfe_{i}]=[g,\bfe_{j}]=0$,  $[g,f(\bfe_{i},\bfe_{j})]$ is identically 0. From the structure of $\fo(n,\CC)$, it follows that 
\[
\left\{\ker\left(\ad_{\bfe_{i}}|_{\fo(n,\CC)}\right)\cap\ker\left(\ad_{\bfe_{j}}|_{\fo(n,\CC)}\right)\right\}^{\perp}=\CC\cdot \bfl_{ij},
\]
where the left hand side denotes the subalgebra of $\fo(n,\CC)$ annihilated by $$\ker\left(\ad_{\bfe_{i}}|_{\fo(n,\CC)}\right)\cap\ker\left(\ad_{\bfe_{j}}|_{\fo(n,\CC)}\right).$$ Thus, $f(\bfe_{i},\bfe_{j})=c\bfl_{ij}$ for some $c\in\CC$ is the only invariant cocycle with image lying in $\fo(n,\CC)$.

Now, let us assume that $f$ is an invariant cocycle with $f(\bfe_{i},\bfe_{j})\in\CC^{n}$, for some $i,j$. We claim that the same holds for any other value of $f$. Indeed, for any $k\neq i,j$, the invariance of $f$ under $\bfl_{ik}$ implies that $[\bfl_{ki},f(\bfe_{i},\bfe_{j})]-f(\bfe_{k},\bfe_{j})$ is equal to a 2-coboundary evaluated at $\bfe_{i}\wedge \bfe_{j}$, and therefore $f(\bfe_{k},\bfe_{j})\in\CC^{n}$. Moreover, for any $l\neq i,j,k$, a similar argument on the invariance of $f$ under $\bfl_{jl}$ shows that $f(\bfe_{k},\bfe_{l})\in\CC^{n}$. In conclusion, the image of an invariant 2-cocycle either lies fully in $\fo(n,\CC)$ or in $\CC^{n}$.

To conclude the classification, let us assume that $f:\bigwedge^{2}\CC^{n}\to\CC^{n}$ is an arbitrary linear map. We claim that $f$ is a 2-coboundary. To see this, observe that the linear space of such maps and the space of linear maps $l:\CC^{n}\to\fo(n,\CC)$ are equidimensional, and the coboundary map $l\mapsto [l(\cdot),\cdot]+[\cdot,l(\cdot)]$ is a linear map between these spaces. A basis for the space of linear maps $l:\CC^{n}\to\fo(n,\CC)$ is given by the set $\{l_{ijk}(\bfe_{m})=\delta_{im}\bfl_{jk}\}_{j<k}$. Since
\[
(dl_{ijk})(\bfe_{m},\bfe_{n})=(\delta_{im}\delta_{jn}-\delta_{in}\delta_{jm})\bfe_{k}+(\delta_{in}\delta_{km}-\delta_{im}\delta_{kn})\bfe_{j},
\]
and the collection of the latter maps is obviously linearly independent, we conclude that the coboundary map is an isomorphism. Therefore, all linear maps $f:\bigwedge^{2}\CC^{n}\to\CC^{n}$ are 2-coboundaries, and any invariant 2-cocycle taking values in $\CC^{n}$ is necessarily trivial in cohomology.

To conclude the proof, it is a routine computation to verify that the bracket deformation of $\fe^{\CC}_{n}$ determined by any choice of nontrivial invariant 2-cocycle,
\[
[\bfe_{i},\bfe_{j}]_{\varepsilon}=\varepsilon\bfl_{ij},\quad \varepsilon\in\CC,
\]
and with the rest of brackets kept the same, satisfies the Jacobi identity and defines a Lie algebra $\fe^{\CC}_{n}(\varepsilon)$ that is isomorphic to $\fo(n+1,\CC)$ if $\varepsilon\neq 0$. In particular, the specialization $\varepsilon\in\RR$ leads to the real forms $\fo(n+1)$ for $\varepsilon>0$, and $\fo(n,1)$ for $\varepsilon<0$. 
\end{proof}

\begin{proof}[\textbf{Proof of theorem \ref{def-Heisenberg}}]
The Heisenberg Lie algebra is defined as a central extension of the abelian Lie algebra $\CC^{2n}$, and $\fo(n,\CC)$ acts trivially in the extension term. This implies that $\fg^{\CC}_{n}=\fo(n,\CC)\ltimes \fh^{\CC}_{n}$ is a central extension of the Lie algebra $\fo(n,\CC)\ltimes \CC^{2n}$, and moreover, we have the following commutative diagram:
\[
\xymatrix{
    {0} \ar[r]  &  {\fh^{\CC}_{n}} \ar[r] \ar[d] & {\fg^{\CC}_{n}}  \ar[r] \ar[d] & {\fo(n,\CC)} \ar[r]\ar@{=}[d] & {0}\\
    {0} \ar[r]  &  {\CC^{2n}} \ar[r]  & {\fo(n,\CC)\ltimes\CC^{2n}} \ar[r]  & {\fo(n,\CC)} \ar[r] & {0} 
}
 \]
We claim that there is a $1-1$ correspondence
\[
\left\{\parbox[d]{1.3in}{\centering Infinitesimal deformations of $\fg^{\CC}_{n}$}\right\} \leftrightarrow
\left\{\parbox[d]{1.15in}{\centering Infinitesimal deformations of $\fo(n,\CC)\ltimes \CC^{2n}$}\right\} 
\]
In other words, there is an isomorphism 
\[
H^{2}\left(\fh^{\CC}_{n},\fg^{\CC}_{n}\right)^{\fo(n,\CC)}\cong H^{2}\left(\CC^{2n},\fo(n,\CC)\ltimes\CC^{2n}\right)^{\fo(n,\CC)}.
\]
It will be convenient to consider a canonical basis for $\fh^{\CC}_{n}$, $\{\bfe_{1},\dots,\bfe_{2n},\bfI\}$, with commutation relations 
\[
[\bfe_{i},\bfe_{n+j}]=\delta_{ij}\bfI,\quad i,j=1,\dots,n. 
\]
Any invariant cocycle $f:\bigwedge^{2}\fh^{\CC}_{n}\to\fg^{\CC}_{n}$ is cohomologous to a cocycle $f'$ satisfying $f(\bfe_{i},\bfI)=0$ $\forall i$. To see this, notice that the cocycle condition implies that $[\bfe_{j},f(\bfe_{i},\bfI)]=0$ $\forall i,j$, therefore $f(\bfe_{i},\bfI)=c_{i}\bfI$. The linear map $l:\fh^{\CC}_{n}\to\fg^{\CC}_{n}$ defined as 
\[
l(\bfe_{j})=0,\qquad l(\bfI)=\sum_{i=1}^{n}(c_{i}\bfe_{n+i}-c_{n+i}\bfe_{i}),
\]
satisfies $(dl)(\bfe_{i},\bfI)=f(\bfe_{i},\bfI)$, and the claim follows. Moreover, a similar argument shows that any invariant cocycle is cohomologous to a cocycle with image in $\fo(n,\CC)\ltimes\CC^{2n}$. Indeed, assume that $f(\bfe_{i},\bfe_{j})=c_{ij}\bfI$. Then the linear map $l:\fh^{\CC}_{n}\to\fg^{\CC}_{n}$ given by $l(\bfe_{k})=\delta_{ik}c_{ij}\bfe_{j-n}$ if $j>n$ (or $=-\delta_{ik}c_{ij}\bfe_{j+n}$ if $j\leq n$) satisfies $(dl)(\bfe_{i},\bfe_{j})=c_{ij}\bfI$ and zero otherwise. This concludes the proof of the isomorphism in cohomology. Thus, it is enough to understand the infinitesimal deformations of $\fo(n,\CC)\ltimes\CC^{2n}$.

Lemma \ref{def-Euclidean} can be used in the classification of the independent invariant 2-cocycles $f:\bigwedge\CC^{2n}\to \fo(n,\CC)\ltimes \CC^{2n}$. A similar argument as in the proof of lemma \ref{def-Euclidean} shows that an invariant 2-cocycle which is not a 2-coboundary necessarily has image in $\fo(n,\CC)$. When thought of as a $\fo(n,\CC)$-module, the subalgebra $\CC^{2n}$ splits as a  direct sum $\CC^{n}_{1}\oplus \CC^{n}_{2}$ of invariant $\fo(n,\CC)$-subspaces. There is an induced splitting 
$\bigwedge^{2}\CC^{2n} = \left(\bigwedge^{2}\CC^{n}_{1}\right)\oplus\left(\CC^{n}_{1}\otimes\CC^{n}_{2}\right)\oplus\left(\bigwedge^{2}\CC^{n}_{2}\right)$, and any invariant 2-cocycle can be decomposed into three different components. The classification problem is then reduced to the classification of invariant 2-cocycles on each component. It follows from lemma \ref{def-Euclidean} that the restriction to $\CC_{1}^{n}$ and $\CC_{2}^{n}$ determines two nontrivial one-dimensional spaces of invariant 2-cocycles, spanned by
\begin{eqnarray}
f_{1}(\bfe_{i},\bfe_{j}) &=& \bfl_{ij}\quad \text{for}\;\; i,j\leq n,\;\;\text{and zero otherwise,} \label{eq:f1}\\
f_{2}(\bfe_{n+i},\bfe_{n+j}) &=& \bfl_{ij}\quad \text{for}\;\;\; i,j\leq n, \;\;\text{and zero otherwise.} \label{eq:f2}
\end{eqnarray}
These are the only possibilities that are supported in the invariant subspaces $\bigwedge^{2}\CC^{n}_{1}$ and $\bigwedge^{2}\CC^{n}_{2}$. The remaining possibility would consist of an invariant 2-cocycle supported in $\CC_{1}^{n}\otimes\CC_{2}^{n}$.
There is an obvious choice, namely
\begin{equation}
f_{3}(\bfe_{i},\bfe_{n+j})=\bfl_{ij}\quad \text{for}\;\;\; i,j\leq n, \;\;\text{and zero otherwise.} \label{eq:f3}
\end{equation}
A similar argument as in the proof of lemma \ref{def-Euclidean} shows that any other invariant 2-cocycle supported in $\CC_{1}^{n}\otimes\CC_{2}^{n}$ must be a multiple of $f_{3}$.
Therefore, any nontrivial invariant 2-cocycle is a linear combination of $f_{1}$, $f_{2}$ and $f_{3}$. 

Let $\varepsilon_{1},\varepsilon_{2},\varepsilon_{3}\in\CC$. The lift to $\fg_{n}$ of the Lie bracket deformations of $\fo(n,\CC)\ltimes \CC^{2n}$ induced by the previous cocycles is determined by
\[
[\bfe_{i},\bfe_{j}]_{\varepsilon_{1}}=\varepsilon_{1}\bfl_{ij},\quad [\bfe_{n+i},\bfe_{n+j}]_{\varepsilon_{2}}=\varepsilon_{2}\bfl_{ij},\quad [\bfe_{i},\bfe_{n+j}]_{\varepsilon_{3}}=\delta_{ij}\bfI+\varepsilon_{3}\bfl_{ij}.
\]
and additionally,
\[
[\bfe_{i},\bfI] = \varepsilon_{3}\bfe_{i}-\varepsilon_{1}\bfe_{n+i},\qquad [\bfe_{n+i},\bfI] = \varepsilon_{2}\bfe_{i}-\varepsilon_{3}\bfe_{n+i},
\]
while the remaining basis elements' Lie brackets are unchanged. It is a routine computation to verify that these Lie bracket deformations satisfy the Jacobi identity. Hence they integrate to a three-dimensional family of deformations $\fg^{\CC}_{n}(\varepsilon_{1},\varepsilon_{2},\varepsilon_{3})$. When $\varepsilon_{1},\varepsilon_{2}\neq 0$, and $\varepsilon_{3}^{2}\neq \varepsilon_{1}\varepsilon_{2}$, there is an isomorphism $\fg_{n}^{\CC}(\varepsilon_{1},\varepsilon_{2},\varepsilon_{3})\cong \fg^{\CC}_{n}(1,1,0) = \fo(n + 2,\CC)$. The details on this isomorphism and the full classification of deformations are described in proposition \ref{prop:def-families}.
\end{proof}

\begin{remark}\label{rem:def-quadratic form}
There is yet another way to describe the linear deformations of $\fg^{\CC}_{n}$, and in particular the invariant 2-cocycles generating $H^{2}\left(\fg^{\CC}_{n},\fg^{\CC}_{n}\right)$, in terms of the geometry of the of the quadratic space $\left(\CC^{n+2},Q^{\CC}_{0}\right)$. Any deformation $Q^{\CC}_{\pmb{\varepsilon}}$ of the quadratic form $Q^{\CC}_{0}$ induces a deformation of the orthogonal Lie algebra $\fo\left(\CC^{n+2},Q^{\CC}_{0}\right) \cong \fg^{\CC}_{n}$, in such a way that if a new quadratic form $\tilde{Q}^{\CC}_{\pmb{\varepsilon}}$ is induced by an orthogonal transformation of $Q^{\CC}_{\pmb{\varepsilon}}$, the corresponding deformations $\fg^{\CC}_{n}(\pmb{\varepsilon})$ and $\tilde{\fg}^{\CC}_{n}(\pmb{\varepsilon})$ are isomorphic. 

Consider the canonical basis $\{\bv_{1},\dots,\bv_{n+2}\}$ of $\CC^{n+2}$, together with the correspondence
\[
\bv_{i}\wedge\bv_{j}\mapsto \bfl_{ij},\quad \bv_{i}\wedge\bv_{n+1}\mapsto \bfe_{i}, \quad \bv_{i}\wedge\bv_{n+2}\mapsto \bfe_{n+i}, \quad \bv_{n+1}\wedge\bv_{n+1}\mapsto \bfI
\]
where $1 \leq i < j \leq n$. It is straightforward to verify that the deformations of the corresponding bilinear form of $Q_{0}$ along the totally isotropic plane $W = \mathrm{Span}\{\bv_{n+1},\bv_{n+2}\}$, parametrized as
\[
(\bv_{n+1},\bv_{n+1}) = \varepsilon_{1},\quad (\bv_{n+2},\bv_{n+2}) = \varepsilon_{2},\quad (\bv_{n+1},\bv_{n+2}) = \varepsilon_{3}
\]
induce the deformation $\fg_{n}^{\CC}(\varepsilon_{1},\varepsilon_{2},\varepsilon_{3})$. Therefore, we conclude \emph{a posteriori} that the map  
\[
\left\{\parbox[d]{1.9in}{\centering Deformations of $Q_{0}$ along $W$}\right\} \rightarrow
\left\{\parbox[d]{2.3in}{\centering  $Z^{2}\left(\fh^{\CC}_{n},\fo(n,\CC)\right)^{\fo(n,\CC)}\cong H^{2}(\fg_{n}^{\CC},\fg_{n}^{\CC})$}\right\}
\]
is a bijection.
\end{remark}

\begin{remark}
The special cases $n = 1, 2$ were excluded from the proof since $\fg_{1} = \fh_{1}$ and $\fo(2)$ is abelian. However, it follows from remark \ref{rem:def-quadratic form} that the spaces $H^{2}(\fg_{n},\fg_{n})$ are still three-dimensional when $n = 1, 2$.
\end{remark}

\section{Deformation of special coadjoint orbits}\label{sec:coadjoint}

Let us assume that $n \geq 3$. 
The three-parameter family of deformations of $\fg_{n}$ can be conveniently prescribed as a deformation of the Lie-Poisson structure on a basis for $\fg_{n}^{\vee}$. Let us consider the dual coordinates $\{l_{ij}\}_{1\leq i<j\leq n}$ in $\fo(n)^{\vee}$ with canonical commutation relations
\begin{equation}\label{eq:comm1}
\{l_{ij},l_{kl}\}=\delta_{ik}l_{jl}-\delta_{il}l_{jk} +\delta_{jl}l_{ik}-\delta_{jk}l_{il}
\end{equation}
together with the Darboux coordinates $\{x_{i},p_{i}\}_{i=1}^{n}$, and the central extension coordinate $I$ in $\fh_{n}^{\vee}$, on which the coordinates $l_{ij}$ act as 
\begin{equation}\label{eq:comm2}
\{l_{ij},x_{k}\}=\delta_{ik}x_{j}-\delta_{jk}x_{i},\quad \{l_{ij},p_{k}\}=\delta_{ik}p_{j}-\delta_{jk}p_{i},\quad \{l_{ij},I\}=0.
\end{equation} 
The commutation relations \eqref{eq:comm1}--\eqref{eq:comm2} do not admit nontrivial deformations as a consequence of the simplicity of $\fo(n)$, and in particular, they are not affected by the integration of the cocycles \eqref{eq:f1}--\eqref{eq:f3}. On the other hand, the linear deformations of the induced Lie--Poisson bracket of the chosen basis for $\fg_{n}^{\vee}$ manifest in the remaining commutation relations. Let $\varepsilon_{1}$, $\varepsilon_{2}$ and $\varepsilon_{3}$ be complex parameters corresponding to the cocycles $f_{1}$, $f_{2}$ and $f_{3}$ in $H^{2}\left(\fg^{\CC}_{n},\fg^{\CC}_{n}\right)$. The Lie--Poisson bracket deformations in $\left(\fg_{n}^{\CC}\right)^{\vee}$ take the explicit form 
\begin{equation}\label{def-comm}
\begin{array}{ll}
\{x_{i},x_{j}\} = \varepsilon_{1}l_{ij}, & \{p_{i},p_{j}\} = \varepsilon_{2}l_{ij},\\\\
 \{x_{i},p_{j}\} = \delta_{ij}I+\varepsilon_{3}l_{ij}, &\\\\
\{x_{i},I\}=\varepsilon_{3}x_{i}-\varepsilon_{1}p_{i}, & \{p_{i},I\}=\varepsilon_{2}x_{i}-\varepsilon_{3}p_{i},
\end{array}
\end{equation}

\begin{remark}
Let $\pmb{\varepsilon} = (\varepsilon_{1},\varepsilon_{2},\varepsilon_{3})$. For every $\lambda \in \CC^{*}$, the nonzero triples $\pmb{\varepsilon} = (\varepsilon_{1},\varepsilon_{2},\varepsilon_{3})$ and $\pmb{\varepsilon}' = \lambda\cdot(\varepsilon_{1},\varepsilon_{2},\varepsilon_{3})$ define isomorphic Lie algebras under scaling of generators. Therefore, in order to describe the different isomorphism classes of nontrivial deformations of $\fg_{n}^{\CC}$, it is sufficient to consider them in terms of a stratification of the projective plane $\PP(\pmb{\varepsilon})$. 
\end{remark}

\begin{remark}
Different special values of nonzero triples $\pmb{\varepsilon}$ determine special Lie algebra deformations. By definition, $\fg^{\CC}_{n}(1,1,0) = \fo(n + 2,\CC)$, which is seen under the relabeling $x_{i} = l_{i n + 1}$, $p_{i} = l_{i n + 2}$, $I = l_{n + 1 n + 2}$. Moreover, the Lie algebras $\fg^{\CC}_{n}(1, 0, 0)$ and $\fg^{\CC}_{n}(0,1,0)$ are isomorphic to $\fo(n + 1,\CC)\ltimes\CC^{n + 1}$ under the respective relabelings  $p_{i} = l_{i n + 1}$ and $x_{i} = l_{i n + 1}$. Finally, let $\fd_{n}$ denote the deformation of $\fh_{n}$  corresponding to the triple $(0,0,1)$. Then $\fd^{\CC}_{n}$ is completely characterized by its ideals $\textrm{Span}\{I, x_{i},p_{i}\}\cong \mathfrak{sl}(2,\CC)$, $i=1,\dots,n$, and $\fg^{\CC}_{n}(0,0,1) = \fo(n,\CC)\ltimes \fd^{\CC}_{n}$. 
Let $\cC \subset\PP(\pmb{\varepsilon})$ be the flat conic defined by the equation 
\[
\varepsilon_{3}^{2} = \varepsilon_{1}\varepsilon_{2},
\]
and for $i = 1, 2, 3$, let 
\[
\cL_{i} = \{\varepsilon_{i} = 0\}\subset\PP(\pmb{\varepsilon}).
\]
Then we have that $\cC\cap \cL_{1} = [0 : 1 : 0]$, $\cC\cap \cL_{2} = [1: 0 : 0]$, and $\cL_{1}\cap\cL_{2} = [0:0:1]$. With the exception of the latter, all such special points belong to $\cL_{3}$.
\end{remark}

\begin{proposition}\label{prop:def-families}
The isomorphism classes of nontrivial deformations $\fg^{\CC}_{n}(\pmb{\varepsilon})$ are stratified in the projective plane $\PP(\pmb{\varepsilon})$ as follows:\\

\begin{enumerate}

\item[(i)] 
$\fg^{\CC}_{n}(\pmb{\varepsilon})\cong \fo(n+2,\CC)$, if $[\pmb{\varepsilon}]\in\cU$, where $\cU$ denotes the Zariski open locus 
\[
\cU = \PP(\pmb{\varepsilon})\setminus\{\cC\cup\cL_{1}\cup\cL_{2}\}. 
\]

\item[(ii)]  $\fg^{\CC}_{n}(\pmb{\varepsilon})\cong
\fo(n+1,\CC)\ltimes\CC^{n+1}$ if $[\pmb{\varepsilon}]\in\cC$.\\

\item[(iii)] $\fg^{\CC}_{n}(\pmb{\varepsilon})\cong \fo(n,\CC)\ltimes\fd^{\CC}_{n}$ if $[\pmb{\varepsilon}]\in\left(\cL_{1}\cup\cL_{2}\right)\setminus \cC$. Notice that\\
\[
\left(\cL_{1}\cup\cL_{2}\right)\setminus \cC= (\cL_{1}\setminus [0 : 1 : 0])\cup (\cL_{2}\setminus [1 : 0 : 0]).
\] 
\end{enumerate}
\end{proposition}

\begin{proof}
The proof follows after a systematic implementation of the following fundamental principle: at a special value of $\pmb{\varepsilon}$, all cocycles that haven't been integrated to a deformation become trivial in cohomology, and hence, the remaining deformations become equivalent to of a linear transformation of the basis elements.  

\noindent (i) Let $\pi_{3}:\PP(\pmb{\varepsilon})\setminus \cL_{1}\cap \cL_{2} \to \cL_{3}$ be the projection $\pi_{3}(\varepsilon_{1},\varepsilon_{2},\varepsilon_{3}) = (\varepsilon_{1},\varepsilon_{2},0)$. If $\pmb{\varepsilon}\in\cU$, we have that $\varepsilon_{1}\neq 0$, $\varepsilon_{2}\neq 0$, and $\varepsilon_{3}^{2}/\varepsilon_{1}\varepsilon_{2}\neq 1$. Then, there is an isomorphism $\fg^{\CC}_{n}(\pi_{3}(\pmb{\varepsilon})) \cong \fg^{\CC}_{n}(\pmb{\varepsilon})$ induced by the linear transformation defined by
\[
x_{i}\mapsto x_{i}+\frac{\lambda\varepsilon_{3}}{2\varepsilon_{2}} p_{i}, \quad p_{i} \mapsto p_{i} + \frac{\lambda\varepsilon_{3}}{2\varepsilon_{1}} x_{i}, \quad I\mapsto \left(1-\frac{\lambda^{2}\varepsilon_{3}^{2}}{4\varepsilon_{1}\varepsilon_{2}}\right)I,\quad l_{ij}\mapsto \lambda l_{ij},
\]
where 
\[
\lambda = \frac{2\varepsilon_{1}\varepsilon_{2}}{\epsilon_{3}^{2}}\left(1-\sqrt{1 - \frac{\varepsilon_{3}^{2}}{\varepsilon_{1}\varepsilon_{2}}}\right) = 1 + O\left(\frac{\varepsilon_{3}^{2}}{\varepsilon_{1}\varepsilon_{2}}\right).
\]
In order to show that $\fg^{\CC}_{n}(\pi_{3}(\pmb{\varepsilon}))\cong \fg^{\CC}_{n}(1,1,0)\cong \fo(n+2,\CC)$, let 
\[
x_{i}=\sqrt{\varepsilon_{1}}l_{in+1}, \qquad p_{i} = \sqrt{\varepsilon_{2}}l_{in+2},\qquad I = \sqrt{\varepsilon_{1}\varepsilon_{2}}l_{n+1n+2}.
\]
\noindent (ii)  Assume $[\pmb{\varepsilon}] \in \cC\setminus\{[1 : 0 : 0],[0 : 1 : 0]\}$. An isomorphism $\fg^{\CC}_{n}(\varepsilon_{1},0,0)\cong \fg^{\CC}_{n}(\pmb{\varepsilon})$ is defined by the linear transformation acting as the identity on $x_{i}$, $I$ and $l_{ij}$, and mapping 
\[
p_{i}\mapsto p_{i}+\sqrt{\frac{\varepsilon_{2}}{\varepsilon_{1}}}x_{i}
\]
An analogous isomorphism can be constructed to show that $\fg^{\CC}_{n}(0,1,0)\cong \fg^{\CC}_{n}(\pmb{\varepsilon})$.

\noindent (iii) Assume that $\pmb{\varepsilon}\in\cL_{1}\setminus [0 : 1 : 0]$. The linear transformation defined by 
\[
p_{i}\mapsto p_{i} -\frac{\varepsilon_{2}}{2\varepsilon_{3}}x_{i}, 
\]
and acting as the identity on $x_{i}$, $I$, and $l_{ij}$ defines the isomorphism $\fg^{\CC}_{n}(\pmb{\varepsilon}) \cong \fg^{\CC}_{n}(0,0,1) = \fo(n,\CC)\ltimes\fd^{\CC}_{n}$. An analogous argument implies the result for any $\pmb{\varepsilon}\in\cL_{2}\setminus [1 : 0 : 0]$.
\end{proof}

\begin{corollary}\label{cor:effective-parameters}
Any deformation $\fg^{\CC}_{n}(\pmb{\varepsilon})$ with $[\pmb{\varepsilon}]\in \cU$ depends only on the two effective parameters $\varepsilon_{1},\varepsilon_{2}$.  Any deformation $\fg^{\CC}_{n}(\pmb{\varepsilon})$ with $[\pmb{\varepsilon}]\in \cC $depends only on one effective parameter (either $\varepsilon_{1}$ or $\varepsilon_{2}$). Any deformation $\fg^{\CC}_{n}(\pmb{\varepsilon})$ with $[\pmb{\varepsilon}]\in \left(\cL_{1}\cup\cL_{2}\right)\setminus \cC$ depends only on the effective parameter $\varepsilon_{3}$.
\end{corollary}

\begin{remark}
From now on, we will assume that the deformation parameters $\pmb{\varepsilon}=(\varepsilon_{1},\varepsilon_{2},\varepsilon_{3})$ are real, unless otherwise stated. 
\end{remark}

In order to describe the different real forms of the deformations $\fg_{n}(\pmb{\varepsilon})$ that arise by restriction to $\RR$, it is necessary to consider instead a stratification of $\RR^{3}\setminus\{(0,0,0)\}$. The lift $\mathrm{pr}^{-1}\left(\cC|_{\RR}\right)\subset\RR^{3}\setminus\{(0,0,0)\}$ has two connected components $\cC_{+}$, $\cC_{-}$, depending on whether $\varepsilon_{1},\varepsilon_{2}\geq 0$ or $\varepsilon_{1},\varepsilon_{2}\leq 0$. $\mathrm{pr}^{-1}\left(\left(\left(\cL_{1}\cup\cL_{2}\right)\setminus\cC\right) |_{\RR}\right)$ will be denoted by $\cL$ (although it possesses two connected components, the corresponding real forms are isomorphic). The set $\RR^{3}\setminus\left\{(\cC_{+}\cup\cC_{-}\cup \cL\right\}$ can be decomposed as 
\[
\cR_{++}\cup\cR_{--}\cup\cR_{+-},
\]
with the regions $\cR_{++}$ and $\cR_{--}$ characterized by the conditions $\varepsilon_{1}, \varepsilon_{2} > 0$ and $\varepsilon_{1}, \varepsilon_{1}<0$, respectively (each region consisting of 3 connected components). The remaining region $\cR_{+-}$ consists of all triples $\pmb{\varepsilon}$ for which either $\varepsilon_{1} > 0$, $\varepsilon_{2} < 0$ or $\varepsilon_{1} < 0$, $\varepsilon_{2} > 0$.

\begin{corollary}\label{cor:real forms}
In terms of the previous stratification of $\RR^{3}\setminus\{(0,0,0)\}$, the isomorphism type of the real forms $\fg_{n}(\pmb{\varepsilon})$ for $\pmb{\varepsilon}\neq (0,0,0)$ is 
\[
\left\{
\begin{array}{rcl}
\fo(n+2) & \text{if} & \pmb{\varepsilon}\in\cR_{++},\\\\
\fo(n+1,1) & \text{if}  & \pmb{\varepsilon}\in\cR_{+-},\\\\
\fo(n,2) & \text{if} & \pmb{\varepsilon}\in\cR_{--},\\\\
 \fo(n+1)\ltimes\RR^{n+1} & \text{if} & \pmb{\varepsilon}\in\cC_{+},\\\\
\fo(n,1)\ltimes \RR^{n+1} & \text{if} & \pmb{\varepsilon}\in\cC_{-},\\\\
\fo(n)\ltimes\mathfrak{d}_{n}|_{\RR} & \text{if} & \pmb{\varepsilon}\in\cL.
\end{array}
\right.
\]
\end{corollary}

\subsection{Special coadjoint orbits}

The rank of a semi-simple Lie algebra $\fg$ is equal to the dimension of the center of its universal enveloping algebra $U(\fg)$---a space generated by the so-called \emph{Casimir invariants}. $Z(U(\fg))$ can be equivalently described in terms of the Lie-Poisson structure in $C^{\infty}(\fg^{\vee})$. For $\fo(n + 2,\CC)$ and the dual basis $\{l_{ij}\}_{1\leq i<j\leq n+2}$ with Poisson brackets \eqref{eq:comm1}, the Casimir invariants can be determined explicitly as the homogeneous polynomials
\[
C_{2k}=\tr\left(\mathrm{L}^{2k}\right),\quad k=1,\dots,\lfloor n/2\rfloor+1,
\]
where $\mathrm{L} = (l_{ij})$. The choice of values for the Casimir invariants determines all the coadjoint orbits of maximal dimension in the orbit stratification of $\fo(n + 2,\CC)^{\vee}$, isomorphic to the quotient of $\mathrm{O}(n+2,\CC)$ by a maximal torus. There is an analogous description of the coadjoint orbits in $\fo(n+2,\CC)^{\vee}$ isomorphic to the homogeneous space $\mathrm{O}(n+2,\CC)/\SO(2,\CC)\times\mathrm{O}(n,\CC)$, and which are the \emph{minimal} nontrivial orbits when $n\neq 2,4$ \cite{Wolf78}. The next result is described in \cite{BS-97}. 

\begin{lemma}[\cite{BS-97}]\label{lemma:minimal}
The $2n$-dimensional coadjoint orbits in $\fo(n+2)^{\vee}$ are isomorphic to the homogeneous space $\mathrm{SO}(n+2)/\mathrm{SO}(2)\times\mathrm{SO}(n)$ and form a 1-dimensional algebraic family determined by the collection of quadratic equations
\begin{equation}
C_{2} = -2r^{2},\label{K2}
\end{equation}
\begin{equation}
l_{i_{1}i_{2}}l_{i_{3}i_{4}} = l_{i_{1}i_{3}}l_{i_{2}i_{4}}-l_{i_{1}i_{4}}l_{i_{2}i_{3}},\quad 1\leq i_{1}<i_{2}<i_{3}<i_{4}\leq n+2.\label{Plucker}
\end{equation}
\end{lemma}

\begin{remark}\label{rem:Gr^{+}_{2}(n+2)}
The set of quadratic equations \eqref{K2}--\eqref{Plucker} identify the given coadjoint orbits with the Grassmannian $\widetilde{\mathrm{Gr}}_{2}\left(\RR^{n + 2}\right)$ of oriented 2-planes in $\RR^{n + 2}$, as they can be understood as a $2:1$ lift of the classical Pl\"ucker embedding. The Pl\"ucker relations indicate that equations \eqref{Plucker} are overdetermined, and can be generated by any subcollection of $n\choose{2}$ equations containing a given fixed element $l_{ij}$, i.e. $l_{n+1n+2}$.
\end{remark}

Applying corollary \ref{cor:real forms} and lemma \ref{lemma:minimal} to the generic deformations $\fg_{n}(\pmb{\varepsilon})$, $\pmb{\varepsilon}\in\RR^{3}$, and letting $x_{i} = l_{i n + 1}$, $p_{i} = l_{i n + 2}$, and $I = l_{n+1 n+2}$,  equation \eqref{K2} becomes
\begin{equation}\label{C1}
C_{2}= -2\left(I^{2}+\varepsilon_{1}x^{2}+\varepsilon_{2}p^{2}-2\varepsilon_{3}xp-(\varepsilon_{3}^{2}
-\varepsilon_{1}\varepsilon_{2})l^{2}\right),
\end{equation}
where 
\[
x^{2} = \sum_{i = 1}^{n} x_{i}^{2},\qquad  p^{2}  = \sum_{i = 1}^{n} p_{i}^{2},\qquad xp = \sum_{i = 1}^{n} x_{i}p_{i}, \qquad l^{2} = \sum_{1 \leq i < j \leq n} l^{2}_{ij}. 
\]
The remaining equations do not depend on the deformation parameters. We emphasize the ones containing $\{I,x_{i},p_{j}\}$, 
\begin{equation}\label{C2}
Il_{ij} = x_{i}p_{j}-x_{j}p_{i},
\end{equation}
\begin{equation}\label{C3}
l_{ij}x_{k} - l_{ik}x_{j} + l_{jk}x_{i} = 0,\qquad l_{ij}p_{k} - l_{ik}p_{j} + l_{jk}p_{i}=0.
\end{equation}
Notice that the subcollection \eqref{C2} generalizes the usual definition of angular momentum and generate \eqref{Plucker}, while equations \eqref{C3} generalize the vector analysis relations $\mathbf{l}\cdot\mathbf{x}=\mathbf{l}\cdot\mathbf{p}=0$.

\begin{definition}
The coadjoint orbits $\OO_{2n}(\pmb{\varepsilon})\subset \fg_{n}(\pmb{\varepsilon})^{\vee}$ are the special $2n$-dimensional orbits defined by the choice of value $C_{2} = -2$ in equation \eqref{C1}.
\end{definition}

\begin{remark}\label{rem:2-1 cover}
It follows from remarks \ref{rem:def-quadratic form} and \ref{rem:Gr^{+}_{2}(n+2)} that over the open set $\cR_{++}$, 
\eqref{C1}--\eqref{C2} correspond to the equations that determine a $2:1$ lift of the Pl\"ucker embedding, identifying $\OO_{2n}(\pmb{\varepsilon})$ with the Grassmannian of oriented planes in 
$\RR^{n+2}$. 
The different solutions of the quadratic equation \eqref{C1} correspond to the different choices of orientation of a given 2-plane in $\RR^{n+2}$. If we consider the degeneration $\pmb{\varepsilon} \to {0}$, the limiting equations 
\begin{equation}\label{eq:orbit-phase-space}
I = \pm 1\qquad \text{and} \qquad \pm l_{ij} = x_{i}p_{j} - x_{j} p_{i},
\end{equation}
define two disjoint orbits $\OO^{+}_{2n}$ and $\OO^{-}_{2n}$ in $\fg_{n}^{\vee}$, each isomorphic to $\RR^{n}\oplus\RR^{n}$. In turn, the degeneration of the canonical symplectic structure determined by the Kirillov--Konstant--Souriau symplectic form \cite{Reyman,Kirillov} on $\OO_{2n}(\pmb{\varepsilon})$ corresponds to the symplectic structure on standard phase space
\[
\left(\RR^{n}\oplus\RR^{n}, \sum_{i = 1}^{n}dx_{i}\wedge dp_{i}\right)
\]
for each of the two orbits in $\fg_{n}^{\vee}$. The existence of two limiting connected components corresponds to the limiting degenerations in the work of Higgs \cite{Higgs78} for the cotangent bundles $T^{*}S^{n}$ as the sectional curvature is allowed to vanish (cf. remark \ref{rem:cotangent}).
\end{remark}

\begin{corollary}\label{cor:deformations}
For any $\pmb{\varepsilon}\in\RR^{3}\setminus\{(0,0,0)\}$, the special coadjoint orbits  $\OO_{2n}(\pmb{\varepsilon})$ 
define a family of deformations of standard phase space, and 
carry a canonical ``angular momentum" representation of $\fo(n)$ in $C^{\infty}(\OO_{2n}(\pmb{\varepsilon}))$ (cf. \cite{LM-03}). These orbits are diffeomorphic to the Grassmannian of oriented 2-planes $\widetilde{\mathrm{Gr}}_{2}\left(\RR^{n + 2}\right)$ 
over the open region $\cR_{++}$. 
 
\end{corollary}

\begin{remark}\label{rem:cotangent}
Over the lift of the flat conic $\mathrm{pr}^{-1}\left(\cC_{\RR}\right) = \cC_{+}\cup\cC_{-}$ (and in particular, in the lines $\varepsilon_{2}=\varepsilon_{3}=0$ and $\varepsilon_{1}=\varepsilon_{3}=0$), 
the coadjoint orbits $\OO_{2n}(\pmb{\varepsilon})$ degenerate to a manifold isomorphic to $T^{*}S^{n}$ if $\pmb{\varepsilon}\in\cC_{+}$ and two copies of $T^{*}\HH^{n}$ if $\pmb{\varepsilon}\in\cC_{-}$. A proof of this fact is given in proposition \ref{prop:const-curv}. Although the symplectic structure inherited in $T^{*}S^{n}$ (resp. $T^{*}\HH^{n}$) is the standard one (see remark \ref{rem:symplectic}), the variables $x_{i}$ and $p_{j}$ do not define Darboux coordinates. Instead, their commutation relations resemble physically the result of adding an external magnetic field in standard phase space \cite{Nov82,Perelomov}. From a physical point of view, the study of dynamical problems over the complete family of coadjoint orbits $\OO_{2n}(\pmb{\varepsilon})$ can also be interpreted as the study of deformations of dynamical systems on $n$-manifolds of constant sectional curvature (which is equal to $\varepsilon_{2}$ when $\varepsilon_{1} = \varepsilon_{3} = 0$).
\end{remark}

\begin{proposition}\label{prop:const-curv}
There is an induced isomorphism $\OO_{2n}(\pmb{\varepsilon})\cong (T^{*}S^{n},\omega)$ over the locus $\cC_{+}$,
where $\omega$ denotes the corresponding standard symplectic form. Over the locus $\cC_{-}$, $\OO_{2n}(\pmb{\varepsilon})$ is a disjoint union of two connected components $\OO^{+}_{2n}(\pmb{\varepsilon})$ and $\OO_{2n}^{-}(\pmb{\varepsilon})$, each symplectomorphic to $(T^{*}\mathbb{H}^{n},\omega)$, corresponding to the values $I>0$ and $I<0$ respectively. 
\end{proposition}

\begin{proof}
It is enough to corroborate this in the case $\varepsilon_{1}=\varepsilon_{3} = 0$; when $\varepsilon_{2} > 0$ (resp. $\varepsilon_{2} < 0$). Then, equation \eqref{C1} determines an $n$-sphere  homogeneous space model in the affine variables $I,x_{i}$ (resp. a two-sheeted $n$-hyperboloid model, with connected components corresponding to the values $I>0$ and $I<0$). Moreover, the orbit $\OO_{2n}(\pmb{\varepsilon})$ of the $\SO(n+1)\ltimes\RR^{n+1}$-action (resp. the $\SO(n,1)\ltimes\RR^{n+1}$-action) on $(\fo(n+1)\ltimes\RR^{n+1})^{*}$  (resp. on $(\fo(n,1)\ltimes\RR^{n+1})^{*}$) determined by equations \eqref{C1}-\eqref{C2} has the structure of a rank-$n$ subbundle $E\to S^{n}$ (resp. two bundles $E^{+}\to\HH^{n}$ and $E^{-}\to\HH^{n}$) of the trivial vector bundle $\fo(n+1)^{*}\times S^{n}$ (resp. two copies of $\fo(n,1)^{*}\times\HH^{n}$), with fiber at a point $(I,x_{1},\dots,x_{n})$ given by the kernel of the map $L_{(I,x_{1},\dots,x_{n})}:\fo(n+1)^{*}\to\fo(n)^{*}$ (resp. $\fo(n,1)^{*}$), defined as
\[
\left(L_{(I,x_{1},\dots,x_{n})}(l,p)\right)_{ij}=Il_{ij}-x_{i}p_{j}+x_{j}p_{i}+\sum_{k=1}^{n}\left(l_{ij}x_{k}-l_{ik}x_{j}+l_{jk}x_{i}\right).
\]
By construction, the bundle of orthonormal frames of $E$ (resp. $E^{+}$ and $E^{-}$) is isomorphic to $\SO(n+1)\to S^{n}$ (resp. $\SO(n,1)\to \HH^{n}$), with fibers corresponding to the isotropy groups of points $(I,x_{1},\dots,x_{n})$ (depending on the values $I>0$ or $I<0$ in the second case), which gives the isomorphism $E\cong T^{*}S^{n}$ (resp. $E^{\pm}\cong T^{*}\HH^{n}$).  
\end{proof}

Observe that on $\cC_{-}$, only the connected component $\OO^{+}_{2n}(\pmb{\varepsilon})$ is of physical significance, as it degenerates to the component $\OO^{+}_{2n}$ corresponding to the value $I=1$ when $\pmb{\varepsilon}\to 0$.

\begin{remark}\label{complex structure}
For any $\pmb{\varepsilon}\in\cR_{++}$,  $\cR_{+-}$, or $\cR_{--}$, the coadjoint orbits $\OO_{2n}(\pmb{\varepsilon})$ are also irreducible Hermitian symmetric spaces, acquiring a natural K\"ahler structure \cite{Bor54}. The tangent space at any point in a given orbit $\OO_{2n}(\pmb{\varepsilon})$ is respectively modeled by one of the quotients $\fm=\fo(n + 2)/\fo(2)\oplus\fo(n)$, $\fo(n + 1,1)/\fo(2)\oplus\fo(n-1,1)$ or $\fo(n,2)/\fo(2)\oplus\fo(n - 2,2)$, and the integrable almost complex structure can be defined as $J=\ad_{\bfI}$, where $\bfI$ is a generator of $\fo(2)\subset\fo(2)\oplus\fo(n - i,i)$, $i = 0, 1, 2$. Therefore, it follows that all orbits $\OO_{2n}(\pmb{\varepsilon})$ possess a natural K\"ahler polarization generalizing the standard complex coordinate polarization determined by
\[
\left(\CC^{n}, \frac{\sqrt{-1}}{2}\sum_{i = 1}^{n} dz_{i} \wedge d\bar{z}_{i}\right),\qquad z_{i} = x_{i} + \sqrt{-1}p_{i}. 
\] 
\end{remark}

\section{Singular real polarizations and free motion}\label{sec:free motion}

By their very definition, the family of coadjoint orbits $\OO_{2n}(\pmb{\varepsilon})$ possess two natural real polarizations, singular over the set $\{I=0\}$, and invariant under a family of groups of symplectomorphisms isomorphic to a deformation of the Euclidean group. These are spanned by the Hamiltonian vector fields corresponding  to the collections of functions $\{x_{i}/I\}$, $\{p_{i}/I\}$ in involution
\[
\{x_{i}/I, x_{j}/I\} = 0,\qquad \{p_{i}/I, p_{j}/I\} = 0,\qquad 1\leq i,j\leq n,
\]
and will be called, respectively, the \emph{position} and \emph{momentum} polarizations.  Both position and momentum polarizations are invariant under a global $\SO(n)$-action, generalizing the standard rotational action in position and momentum coordinates, and which is characterized infinitesimally by the momentum map 
\[
\lambda:\OO_{2n}(\pmb{\varepsilon})\to \fo(n)^{*}, \qquad(\lambda)_{ij}=l_{ij}.
\]
Let $q_{i}=x_{i}/I$, $i=1,\dots,n$. The choice of the position polarization motivates the introduction of a family of functions playing the role of the free-motion Hamiltonians in deformed phase space, namely
\begin{equation}
H_{0}(\pmb{\varepsilon}) = \frac{1}{2}\left(p^{2}+\varepsilon_{2} l^{2}\right) = \frac{1}{2}\left(p^{2}+\varepsilon_{2} \left(p^{2}q^{2}-(pq)^{2}\right)\right),\label{eq:free-motion}
\end{equation} 
and which posses two equivalent geometric interpretations in terms of the dynamical symmetries of a family of ${n + 1 \choose 2 }$-dimensional Lie groups. They not only coincide with the quadratic Casimir invariants of the Lie subalgebras spanned by the dual elements $\{l_{ij}\}$ and $\{p_{k}\}$, but also correspond to $|\mu_{0}(\pmb{\varepsilon})|^{2}$, the square of the norm of a family of momentum maps
\[
\mu_{0}(\pmb{\varepsilon}):\OO_{2n}(\pmb{\varepsilon}) \to\fk(\pmb{\varepsilon})^{*}
\]
where 
\[
\fk(\pmb{\varepsilon})\cong \left\{ 
\begin{array}{cl}
\fo(n+1) &  \text{if}\quad \varepsilon_{2} >0,\\\\ 
\fo(n,1) &   \text{if}\quad \varepsilon_{2} < 0,\\\\ 
\fe_{n+1} &  \text{if} \quad \varepsilon_{2} = 0.
\end{array}
\right.
\]
Thus, $H_{0}(\pmb{\varepsilon})$ and $\mu_{0}(\pmb{\varepsilon})$ respectively generalize the standard free-motion Hamiltonian and the corresponding Euclidean group momentum map in $(\RR^{n}\oplus\RR^{n},\omega)$. In particular, over the contraction $\varepsilon_{1}=\varepsilon_{3}=0$, the coordinates $\{q_{i}\}$ correspond to the \emph{gnomonic coordinates} over the $n$-sphere \cite{Higgs78} if $\varepsilon_{2} > 0$ and hyperbolic $n$-space if $\varepsilon_{2} < 0$ and $I > 0$, while $H_{0}$ corresponds to the Hamiltonian inducing geodesic motion.

\begin{remark}\label{rem:symplectic}
In the coordinates $\{q_{i},p_{j}\}$, defined over the open set $\{I\neq 0\}$, Kirillov's symplectic form on the family $\OO_{2n}(0,\varepsilon_{2},0)$ takes the simple form
\begin{equation}\label{local-2-form}
\omega_{\pmb{\varepsilon}} = -d\theta_{\pmb{\varepsilon}},\qquad \theta_{\pmb{\varepsilon}} = \sum_{i=1}^{n}\left(p_{i}-\varepsilon_{2}\frac{(q,p)q_{i}}{1+\varepsilon_{2}q^{2}}\right)dq_{i}.
\end{equation}
The analogous expression over the real flat conic $\varepsilon_{3}^{2}=\varepsilon_{1}\varepsilon_{2}$ can then be reconstructed by means of a suitable linear transformation (see proposition \ref{prop:def-families}). In particular, when $\varepsilon_{1}=\varepsilon_{3}=0$, $\varepsilon_{2}>0$ (resp. $\varepsilon_{2}<0$),  the above explicit expression for the Liouville form $\theta_{\pmb{\varepsilon}}$ provides the standard Darboux coordinates with conjugated momenta 
\[
p_{i}-\varepsilon_{2}\frac{(q,p)q_{i}}{1+\varepsilon_{2}q^{2}}
\]
on $T^{*}S^{n}$,  (resp. $T^{*}\HH^{n}$ when either $I>0$ or $I<0$). 
\end{remark}

\appendix

\section{Standard phase space as a coadjoint orbit}\label{appendix}

\begin{proposition}\label{prop:App1}
There is a symplectomorphism 
\[
\left(\mathbb{R}^{n}\oplus\mathbb{R}^{n},\omega = \sum_{i = 1}^{n} dx_{i}\wedge dp_{i}\right)\cong \OO^{+}_{2n} 
\]
to a connected component of a $2n$-dimensional coadjoint orbit $\OO_{2n}$ of the group $\mathrm{O}(n)\ltimes\mathrm{H}_{n}$, mapping the standard $\mathrm{O}(n)$-action on $\RR^{n}\oplus\RR^{n}$ to the corresponding coadjoint action on $\OO_{2n}$.
\end{proposition}

\begin{proof}
It is convenient to identify a suitable set of generators and relations on the dual space $\fg^{\vee}_{n}$. Let $\{x_{1},p_{1},\dots,x_{n},p_{n},I\}$ be a set of standard dual variables for the Heisenberg Lie algebra $\fh_{n}$, and let $\{l_{ij}\}_{1\leq i < j \leq n}$ be dual variables for the orthogonal Lie algebra $\fo(n)$. By definition, the symplectic structure of a coadjoint orbit is determined by the Lie-Poisson bracket on $C^{\infty}(\fg_{n}^{\vee})$. Consider the $2n$-dimensional coadjoint orbit $\OO_{2n}\subset \fg_{n}^{\vee}$ determined by fixing the values $I=\pm1$, together with the \emph{angular momentum relations}
\[
l_{ij} = x_{i}p_{j} - x_{j}p_{i}, \qquad 1 \leq i < j \leq n.
\]
The canonical commutation relations $\{x_{i},p_{j}\} = \delta_{ij}$ will follow if we restrict to the connected component $\OO^{+}_{2n}$ given by $I = 1$. The correspondence of symplectic $\mathrm{O}(n)$-actions readily follows.
\end{proof}

\begin{proposition}\label{prop:App2}
Let  $Q_{0}$ denote the quadratic form in $\RR^{n+2}$ prescribed by $$Q_{0}(a_{1},\dots, a_{n+2}) = a_{1}^{2} + \dots + a_{n}^{2}.$$ There is an isomorphism $\fg_{n}\cong \fo\left(\RR^{n+2},Q_{0}\right)$. There is an induced diffeomorphism between $\OO_{2n}$ and the Zariski open subset $\cU\subset \widetilde{\mathrm{Gr}}_{2}\left(\RR^{n+2}\right)$ consisting of oriented 2-planes $P\subset\RR^{n+2}$ such that $a_{n+1}\wedge a_{n+2}|_{P}\not\equiv 0$, i.e., whose image under the Pl\"ucker embedding lies in the complement of the zero locus 
\[
Z(a_{n+1}\wedge a_{n+2})\subset \PP\left(\bigwedge^{2}\RR^{n+2}\right).
\]
\end{proposition}

\begin{proof}
Recall that any bilinear form $(\cdot,\cdot)$ on a vector space $V$ induces a Lie algebra structure on $\bigwedge^{2} V$ in terms of the orthogonal endomorphisms 
\[
v\wedge w \mapsto L_{v\wedge w}(u) := (v,u)w - (w,u)v, 
\]
\cite{GS84,GS06}. A direct computation shows that the Lie algebra structure on $\bigwedge^{2} \RR^{n+2}$ induced by the quadratic form $Q_{0}(a_{1},\dots, a_{n+2}) = a_{1}^{2} + \dots + a_{n}^{2}$ is isomorphic to $\fg_{n}$ under the dual correspondence 
\[
a_{i} \wedge a_{j} \mapsto l_{ij} \quad 1\leq i < j \leq n, 
\]
\[
a_{i}\wedge a_{n + 1} \mapsto  x_{i},\quad a_{i}\wedge a_{n + 2} \mapsto p_{i},\quad 1 \leq i \leq n,\quad  a_{n+1}\wedge a_{n+2}\mapsto I.
\]
Such a correspondence identifies the angular momentum relations along the hyperplanes $I = \pm 1$ in $\fg_{n}$ with the Pl\"ucker relations along the hyperplanes $a_{n+1}\wedge a_{n+2} = \pm 1$. 
Let  $\mathrm{pr}:\widetilde{\mathrm{Gr}}_{2}\left(\RR^{n+2}\right)\rightarrow \mathrm{Gr}_{2}\left(\RR^{n+2}\right)$ the projection forgetting orientation, $\iota:\mathrm{Gr}_{2}\left(\RR^{n+2}\right)\hookrightarrow \PP\left(\bigwedge^{2}\RR^{n+2}\right)$ the classical Pl\"ucker embedding, and let 
\[
\cU = \widetilde{\mathrm{Gr}}_{2}\left(\RR^{n+2}\right)\setminus \left(\mathrm{pr}\circ\iota\right)^{-1}\left(Z(a_{n+1}\wedge a_{n+2})\right)
\] 
Since the level sets $a_{n+1}\wedge a_{n+2} = \pm 1$ in $\bigwedge^{2}\RR^{n+2}$ determine uniquely a choice of orientation in every 2-plane in $\mathrm{Gr}_{2}\left(\RR^{n+2}\right)\setminus \iota^{-1}\left(Z(a_{n+1}\wedge a_{n+2})\right)$, we conclude that there is an induced diffeomorphism $\cU\cong \OO_{2n}$. 
 In particular the stabilizer in $\mathrm{G}_{n}$ of the unique totally isotropic 2-plane $W$ in $\RR^{n+2}$ is equal to $\mathrm{O}(n)\times\RR\subset \mathrm{O}\left(\RR^{n+2},Q_{0}\right)\cong \mathrm{G}_{n}$, with the $\RR$-factor corresponding to central extension elements (cf. remark \ref{rem:2-1 cover}). 
\end{proof}

\begin{remark}
Under the correspondences described in propositions \ref{prop:App1} and \ref{prop:App2}, the Howe pair $( \mathrm{O}(n),\mathrm{SL}(2,\RR))$ of $(\mathbb{R}^{n}\oplus\mathbb{R}^{n},\omega)$ \cite{KKS78,Howe89} is induced by the maximal compact subgroup $G\subset\mathrm{O}\left(\RR^{n+2},Q_{0}\right)\cong \mathrm{G}_{n}$ and the group $G'$ of endomorphisms of the unique totally isotropic plane in $\left(\RR^{n+2},Q_{0}\right)$ preserving the area element $ a_{n+1}\wedge a_{n+2} = 1$.
\end{remark}

\subsection*{Acknowledgments} 
I would like to thank Jacob Mostovoy for introducing me to the ideas that led to this work and for multiple discussions, Rolf Farnsteiner for clarifying a question on spectral sequences, and the referees for their constructive criticism and important remarks. The majority of this work was developed under a postdoctoral fellowship at the MPIM in Bonn in the Spring of 2015. While concluding it the author was supported by the DFG SPP 2026 priority programme ``Geometry at infinity".

\bibliographystyle{amsalpha}
\bibliography{Deformations}

\providecommand{\bysame}{\leavevmode\hbox to3em{\hrulefill}\thinspace}
\providecommand{\MR}{\relax\ifhmode\unskip\space\fi MR }
\providecommand{\MRhref}[2]{%
  \href{http://www.ams.org/mathscinet-getitem?mr=#1}{#2}
}
\providecommand{\href}[2]{#2}
\begin{thebibliography}{RSTS94}

\bibitem[Bor54]{Bor54}
A.~Borel, \emph{{K}\"ahlerian coset spaces of semisimple {L}ie groups}, Proc.
  Nat. Acad. Sci. U. S. A. \textbf{40} (1954), 1147--1151.

\bibitem[BS97]{BS-97}
O.~M. Boyarskyi and T.~V. Skrypnik, \emph{{D}egenerate orbits of adjoint
  representation of orthogonal and unitary groups regarded as algebraic
  submanifolds}, Ukrainian Math. J. \textbf{49} (1997), no.~7, 1003--1015.

\bibitem[CE48]{CE48}
C.~Chevalley and S.~Eilenberg, \emph{{C}ohomology theory of {L}ie groups and
  {L}ie algebras}, Trans. Amer. Math. Soc. \textbf{63} (1948), no.~1, 85--124.

\bibitem[Fia85]{Fia85}
A.~Fialowski, \emph{{D}eformations of {L}ie algebras}, Mat. {S}b. ({N}. {S}.)
  \textbf{127(169)} (1985), no.~4, 476--482.

\bibitem[Fra65]{Fra65}
D.~M. Fradkin, \emph{Three-dimensional isotropic harmonic oscillator and
  ${SU}_{3}$}, Am. J. Phys. \textbf{33} (1965), no.~3, 207--211.

\bibitem[Ger64]{Ger64}
M.~Gerstenhaber, \emph{{O}n the deformation of rings and algebras}, Ann. Math.
  \textbf{79} (1964), no.~1, 59--103.

\bibitem[GS84]{GS84}
V.~Guillemin and S.~Sternberg, \emph{{S}ymplectic techniques in physics},
  {C}ambridge {U}niversity {P}ress, 1984.

\bibitem[GS06]{GS06}
\bysame, \emph{{V}ariations on a theme by {K}epler}, {C}olloquium
  {P}ublications Vol. 42, {A}merican {M}athematical {S}oc., 2006.

\bibitem[Hig79]{Higgs78}
P.~W. Higgs, \emph{{D}ynamical symmetries in a spherical geometry {I}}, J.
  Phys. A \textbf{12} (1979), no.~3, 309--323.

\bibitem[How89]{Howe89}
R.~Howe, \emph{Remarks on classical invariant theory}, Trans. Amer. Math. Soc.
  \textbf{313} (1989), no.~2, 539--570.

\bibitem[HS53]{HS-53}
G.~Hochschild and J-P. Serre, \emph{Cohomology of {L}ie algebras}, Ann. Math.
  (1953), 591--603.

\bibitem[IW53]{IW53}
E.~Inonu and E.~P. Wigner, \emph{On the contraction of groups and their
  representations}, Proc. Nat. Acad. Sci. U. S. A. \textbf{39} (1953),
  510--524.

\bibitem[Kir04]{Kirillov}
A.~A. Kirillov, \emph{{L}ectures on the orbit method}, {G}raduate {S}tudies in
  {M}athematics, {V}ol. 64, {A}merican {M}athematical {S}oc., 2004.

\bibitem[KKS78]{KKS78}
D.~Kazhdan, B.~Kostant, and S.~Sternberg, \emph{{H}amiltonian group actions and
  dynamical systems of {C}alogero type}, Comm. Pure Appl. Math. \textbf{31}
  (1978), no.~4, 481--507.

\bibitem[Len24]{Len24}
W.~Lenz, \emph{\"{U}ber den {B}ewegungsverlauf und die {Q}uantenzust\"ande der
  gest\"orten {K}eplerbewegung}, Z. Phys. \textbf{24} (1924), no.~1, 197--207.

\bibitem[LM03]{LM-03}
A.~Leznov and J.~Mostovoy, \emph{{C}lassical dynamics in deformed spaces}, J.
  Phys. A \textbf{36} (2003), no.~5, 1439--1449.

\bibitem[LN67]{LN67}
M.~Levy-Nahas, \emph{{D}eformation and contraction of {L}ie algebras}, J. Math.
  Phys. \textbf{8} (1967), no.~6, 1211--1222.

\bibitem[Nov82]{Nov82}
S.~P. Novikov, \emph{The {H}amiltonian formalism and a multivalued analogue of
  {M}orse theory}, Uspekhi Mat. Nauk \textbf{37} (1982), no.~5(227), 3--49,
  248, English translation in Russian Math. Surveys \textbf{37}:5 (1982),
  1--56.

\bibitem[Per90]{Perelomov}
A.~Perelomov, \emph{{I}ntegrable systems of classical mechanics and {L}ie
  algebras}, {B}irkh\"auser, 1990.

\bibitem[RSTS94]{Reyman}
A.~Reyman and M.~A. Semenov-Tian-Shansky, \emph{{G}roup-theoretical methods in
  the theory of finite-dimensional integrable systems}, Dynamical systems VII,
  Springer Berlin Heidelberg, 1994.

\bibitem[Wei95]{Weibel}
C.~A. Weibel, \emph{{A}n introduction to homological algebra}, {S}tudies in
  {A}dvanced {M}athematics, {V}ol. 38, {C}ambridge {U}niversity {P}ress, 1995.

\bibitem[Wol78]{Wolf78}
J.~Wolf, \emph{{R}epresentations associated to minimal co-adjoint orbits},
  Differential Geometrical Methods in Mathematical Physics II. Springer Berlin
  Heidelberg (1978), 329--349.

\end{thebibliography}
\end{document}